\documentclass{amsart}

\usepackage{amsmath, amsthm, amscd, amsfonts, amssymb, graphicx, color}

\makeatletter
\renewcommand\@seccntformat[1]{\csname the#1\endcsname.\quad}

\newtheorem{theorem}{Theorem}[section]

\newtheorem{corollary}[theorem]{Corollary}

\newtheorem{example}[theorem]{Example}

\newtheorem{lemma}[theorem]{Lemma}

\newtheorem{proposition}[theorem]{Proposition}
\newtheorem{remark}[theorem]{Remark}

\newcommand{\norm}[3]{\ensuremath{\left\Vert#1\right\Vert_{#2}^{#3}}}
\newcommand{\abs}[3]{\ensuremath{\left\vert#1\right\vert_{#2}^{#3}}}

\begin{document}

\setcounter{page}{1}

\author[M. Ciesielski]{Maciej Ciesielski}

\address{Institute of Mathematics, Pozna\'{n} University of Technology, Piotrowo 3A, 60-965 Pozna\'{n}, Poland}
\email{{maciej.ciesielski@put.poznan.pl}}

\title[lower and upper local uniform $K$-monotonicity]{lower and upper local uniform $K$-monotonicity in symmetric spaces}

\begin{abstract}
Using the local approach to the global structure of a symmetric space $E$ we establish a relationship between strict $K$- monotonicity, lower (resp. upper) local uniform $K$-monotonicity, order continuity and the Kadec-Klee property for global convergence in measure. We also answer the question under which condition upper local uniform $K$-monotonicity concludes upper local uniform monotonicity. Finally, we present a correlation between $K$-order continuity and lower local uniform $K$-monotonicity in a symmetric space $E$ under some additional assumptions on $E$.  

\end{abstract}

\maketitle

\bigskip\ 

{\small \underline{2010 Mathematics Subjects Classification: 46E30, 46B20, 46B42.}\hspace{1.5cm}\
	\ \ \quad\ \quad  }\smallskip\ 

{\small \underline{Key Words and Phrases:}\hspace{0.15in} Symmetric space, Lorentz space, $K$-order continuity, lower (upper) local uniform $K$-monotonicity, the Kadec-Klee property for global convergence in measure.}

\bigskip\ \ 

\section{Introduction}

The first essential result devoted to upper local uniform $K$- monotonicity ($ULUKM$) was published in \cite{ChDSS} by Chilin, Dodds, Sedaev, and Sukochev in 1996. Authors presented a complete characterization of $ULUKM$ written in terms of strict $K$-monotonicity and the Kadec-Klee property for global convergence in measure in symmetric spaces, among others. Recently, many interesting results have appeared in \cite{chkm-geom-positivity,Cies-geom,CieKolPlu,CieKolPluSKM,hkm-geo-prop}, where there have been explored the global and local $K$-monotonicity structure of Banach spaces. 

The crucial inspiration for our discussion was found in paper \cite{Cies-JAT}, where there has been studied an application of strict $K$-monotonicity and $K$-order continuity to the best dominated approximation with respect to the Hardy-Littlewood-P\'olya relation $\prec$. It is worth mentioning that in view of the previous result, in \cite{Cies-SKM-KOC} there has been investigated, among others, a full criteria for $K$-order continuity in symmetric spaces.  

The main goal of this manuscript is an investigation dedicated to a complete characterization of strict $K$-monotonicity and $K$-order continuity as well as upper and lower local uniform $K$-monotonicity in symmetric spaces. We organize the paper in the following way. Preliminaries contain all necessary definitions and notions. 
 
In the section 3 we focus on a characterization of lower and upper local uniform $K$-monotonicity in symmetric space $E$. First, we investigate a relation between a point of lower local uniform $K$-monotonicity and a point of lower local uniform monotonicity. We also characterize a full correlation between a point of lower local uniform $K$-monotonicity and a conjunction of a point of order continuity and a point of lower $K$-monotonicity and also an $H_g$ point in a symmetric space $E$. Next, we show a correspondence between a point of upper local uniform $K$-monotonicity and a point of upper local uniform monotonicity and also an $H_g$ point in $E$ under some additional assumption. In our investigation we don't restrict ourself only to the local approach to $K$-monotonicity structure, but we also discuss as a consequence a complete characterization of global $K$-monotonicity properties in a symmetric space $E$. We answer the crucial question under which condition lower local uniform $K$-monotonicity and upper local uniform $K$-monotonicity coincide in symmetric spaces. In the spirit of the previous result, we also describe an essential connection between a point of $K$-order continuity and a point of lower local uniform $K$-monotonicity and also an $H_g$ point in a symmetric space $E$. It is worth mentioning that several results and examples concerning respective global properties are also presented in this section.

\section{Preliminaries}

Let $\mathbb{R}$, $\mathbb{R}^+$ and $\mathbb{N}$ be the sets of reals, nonnegative reals and positive integers, respectively. In a Banach space $(X,\norm{\cdot}{X}{})$ we use a notation $S(X)$ (resp. $B(X))$ for the unit sphere (resp. the closed unit ball). A nonnegative mapping $\phi$ given on $\mathbb{R}^+$ is called \textit{quasiconcave} if $\phi(t)$ is increasing and $\phi(t)/t$ is decreasing on $\mathbb{R}^+$ and also $\phi(t)=0\Leftrightarrow{t=0}$. Denote as usual by $\mu$ the Lebesgue measure on $I=[0,\alpha)$, where $\alpha =1$ or $\alpha =\infty$, and by $L^{0}$ the set of all (equivalence classes of) extended real valued Lebesgue measurable functions on $I$. We also use the notation $A^{c}=I\backslash A$ for any measurable set $A$. Let us recall that a Banach lattice $(E,\Vert \cdot \Vert _{E})$ is said to be a \textit{Banach
function space} (or a \textit{K\"othe space}) if it is a sublattice of $L^{0}$ satisfying the following conditions
\begin{itemize}
\item[(1)] If $x\in L^0$, $y\in E$ and $|x|\leq|y|$ a.e., then $x\in E$ and $%
\|x\|_E\leq\|y\|_E$.
\item[(2)] There exists a strictly positive $x\in E$.
\end{itemize}
In addition, we employ in our investigation the symbol $E^{+}={\{x \in E:x \ge 0\}}$. 

Given $x\in E$ is said to be a \textit{point of order continuity} if for any
sequence $(x_{n})\subset{}E^+$ with $x_{n}\leq \left\vert x\right\vert 
$ and $x_{n}\rightarrow 0$ a.e. we have $\left\Vert x_{n}\right\Vert
_{E}\rightarrow 0.$ A Banach function space $E$ is called \textit{order continuous 
}(shortly $E\in \left( OC\right) $) if any element $x\in{}E$ is a point of order continuity (see \cite{LinTza}). It is said that a Banach function space $E$ has the \textit{Fatou property} whenever for every $\left( x_{n}\right)\subset{}E^+$, $\sup_{n\in \mathbb{N}}\Vert x_{n}\Vert
_{E}<\infty$ and $x_{n}\uparrow x\in L^{0}$ we have $x\in E$ and $\Vert x_{n}\Vert _{E}\uparrow\Vert x\Vert
_{E}$. In addition, we assume that $E$ has the Fatou property, unless it is mentioned otherwise.

An element $x\in E^{+}$ is called a \textit{point of upper local uniform monotonicity} (resp. a \textit{point of lower local uniform monotonicity}) shortly a $ULUM$ \textit{point} (resp. an $LLUM$ \textit{point}) if for any $(x_{n})\subset E$ such that $x\leq x_{n}$ and $\left\Vert{}x_{n}\right\Vert _{E}\rightarrow \left\Vert x\right\Vert _{E}$ (resp. $x_n\leq{x}$ and $\norm{x_n}{E}{}\rightarrow\norm{x}{E}{}$), we get $\left\Vert x_{n}-x\right\Vert _{E}\rightarrow 0$. Let us recall that if each
point of $E^{+}\setminus \left\{ 0\right\} $ is a $ULUM$ point (resp. an $LLUM$ point), then we say that $E$ is \textit{upper local uniformly monotone} shortly $E\in \left(ULUM\right)$ (resp. \textit{lower local uniformly monotone} shortly $E\in \left(LLUM\right)$).

Given $x\in{E}$ is said to be an $H_g$ \textit{point} (resp. an $H_l$ \textit{point}) in $E$ if for any sequence $(x_n)\subset{E}$ with $x_n\rightarrow{x}$ globally in measure (resp. locally in measure) and $\norm{x_n}{E}{}\rightarrow\norm{x}{E}{}$, then $\left\Vert x_n-x\right\Vert _{E}\rightarrow{0}$. Let us recall that the space $E$ has the \textit{Kadec-Klee property for global convergence in measure} (resp. \textit{Kadec-Klee property for local convergence in measure}) if any element $x\in{E}$ is an $H_g$ point (resp. an $H_l$ point) in $E$ (see \cite{ChDSS,CieKolPlu}). 

For any function $x\in L^{0}$ we define its \textit{distribution function} by 
\begin{equation*}
d_{x}(\lambda) =\mu\left\{ s\in [ 0,\alpha) :\left\vert x\left(s\right) \right\vert >\lambda \right\},\qquad\lambda \geq 0.
\end{equation*}
The \textit{decreasing rearrangement} for any element $x\in L^{0}$ is given by 
\begin{equation*}
x^{\ast }\left( t\right) =\inf \left\{ \lambda >0:d_{x}\left( \lambda
\right) \leq t\right\}, \text{ \ \ } t\geq 0.
\end{equation*}
In the whole paper, it is used the notation $x^{*}(\infty)=\lim_{t\rightarrow\infty}x^{*}(t)$ if $\alpha=\infty$ and $x^*(\infty)=0$ if $\alpha=1$. For any function $x\in L^{0}$ we denote the \textit{maximal function} of $x^{\ast }$ by 
\begin{equation*}
x^{\ast \ast }(t)=\frac{1}{t}\int_{0}^{t}x^{\ast }(s)ds.
\end{equation*}
Let us mention that for any function $x\in L^{0}$ it is well known that $x^{\ast }\leq x^{\ast \ast },$ $x^{\ast \ast }$ is decreasing, continuous and subadditive. For more details of $d_{x}$, $x^{\ast }$ and $x^{\ast \ast }$ see \cite{BS, KPS}. 

We say that two functions $x,y\in{L^0}$ are said to be \textit{equimeasurable} (shortly $x\sim y$) if $d_x=d_y$. A Banach function space $(E,\Vert \cdot \Vert_{E}) $ is called \textit{symmetric} or \textit{rearrangement invariant} (r.i. for short) if for any $x\in L^{0}$ and $y\in E$ with $x\sim y$, we have $x\in E$ and $\Vert x\Vert_{E}=\Vert y\Vert _{E}$. In a symmetric space $E$ we denote by $\phi_{E}$ the \textit{fundamental function} given by $\phi_{E}(t)=\Vert\chi_{(0,t)}\Vert_{E}$ for any $t\in [0,\alpha)$ (see \cite{BS}). For any two functions $x,y\in{}L^{1}+L^{\infty }$ it is defined the \textit{Hardy-Littlewood-P\'olya relation} $\prec$ by 
\begin{equation*}
x\prec y\Leftrightarrow x^{\ast \ast }(t)\leq y^{\ast \ast }(t)\text{ for
all }t>0.\text{ }
\end{equation*}

A symmetric space $E$ is called $K$-\textit{monotone} (shortly $E\in(KM)$) if for any $x\in L^{1}+L^{\infty}$ and $y\in E$ with $x\prec y,$ we have $x\in E$ and $\Vert x\Vert_{E}\leq \Vert y\Vert _{E}.$ It is well known that a symmetric space is $K$-monotone if and only if $E$ is exact interpolation space between $L^{1}$ and $L^{\infty }.$ It is worth mentioning that a symmetric space $E$ equipped with an order continuous norm or with the Fatou property is $K$-monotone (see \cite{KPS}).

An element $x\in{E}$ is said to be a \textit{point of lower $K$-monotonicity} shortly an $LKM$ \textit{point} of $E$ if for any $y\in{E}$, $x^*\neq{y^*}$ and $y\prec{x}$, then $\norm{y}{E}{}<\norm{x}{E}{}$. Let us mention that a symmetric space $E$ is called \textit{strictly $K$-monotone} (shortly $E\in(SKM)$) if any element of $E$ is an $LKM$ point. 

An element $x\in{E}$ we call a \textit{point of $K$-order continuity} of $E$ if for any sequence $(x_n)\subset{E}$ with $x_n\prec{x}$ and $x_n^*\rightarrow{0}$ a.e. we have $\norm{x_n}{E}{}\rightarrow{0}$. Recall that a symmetric space $E$ is said to be $K$-\textit{order continuous} (shortly $E\in\left(KOC\right)$) if every element $x$ of $E$ is a point of $K$-order continuity.

An element $x\in{E}$ is said to be a \textit{point of upper local uniform $K$-monotonicity} of $E$ (shortly a $ULUKM$ \textit{point}) if for any $(x_n)\subset{E}$ such that $x\prec{x_n}$ for every $n\in\mathbb{N}$ and $\norm{x_n}{E}{}\rightarrow\norm{x}{E}{}$, then $\norm{x^*-x_n^{*}}{E}{}\rightarrow{0}$. Given a point $x\in{E}$ is said to be a \textit{point of lower local uniform $K$-monotonicity} of $E$ (shortly an $LLUKM$ point) if whenever for any $(x_n)\subset{E}$ with $x_n\prec{x}$ for all $n\in\mathbb{N}$ and  $\norm{x_n}{E}{}\rightarrow\norm{x}{E}{}$, we have $\norm{x^*-x_n^{*}}{E}{}\rightarrow{0}$. A symmetric space $E$ is said to be \textit{upper local uniformly $K$-monotone} shortly $E\in(ULUKM)$ (resp. \textit{lower local uniformly $K$-monotone} shortly ($E\in(LLUKM)$) if whenever every element of $E$ is a $ULUKM$ point (resp. an $LLUKM$ point). For more details  we encourage to see \cite{ChDSS,Cies-geom,Cies-JAT,Cies-SKM-KOC,hkm-geo-prop}.

Recall that the Marcinkiewicz function space $M_{\phi}^{(*)}$ (resp. $M_{\phi}$), where $\phi$ is a quasiconcave function on $I$, is a subspace of $L^0$ such that for all $x\in{M_{\phi}^{(*)}}$ (resp. $x\in{M_\phi}$),
$$\norm{x}{M_{\phi}^{(*)}}{}=\sup_{t>0}\{x^*(t)\phi(t)\}<\infty$$
$${\left(\textnormal{resp. }\norm{x}{M_{\phi}}{}=\sup_{t>0}\{x^{**}(t)\phi(t)\}<\infty\right).}$$
Obviously,  $\norm{x}{M_{\phi}^{(*)}}{}\leq\norm{x}{M_\phi}{}$ for all $x\in{}M_{\phi}$, i.e. the embedding of $M_{\phi}$ in $M_{\phi}^{(*)}$ has norm $1$ (shortly $M_{\phi}{\hookrightarrow}M_{\phi}^{(*)}$). Moreover, it is necessary to mention that the Marcinkiewicz space $M_{\phi}^{(*)}$ (resp. $M_{\phi}$) is a r.i. quasi-Banach function space (r.i. Banach function space) with the fundamental function $\phi$ on $I$. Let us also recall that for any symmetric space $E$ with the fundamental function $\phi$ we have $E{\hookrightarrow}M_{\phi}$ the embedding with norm $1$ (see \cite{BS,KPS}). 

For given $0<p<\infty$ and a locally integrable weight function $w\geq{0}$ we define the Lorentz space $\Lambda_{p,w}$ as a subspace of $L^0$ such that
\begin{equation*}
\left\Vert x\right\Vert _{\Lambda_{p,w}}=\left( \int_{0}^{\alpha}(x^{\ast }(t))^{p}w(t)dt\right) ^{1/p}<\infty,
\end{equation*}
where $W(t)=\int_{0}^{t}w<\infty$ for any $t\in{I}$ and $W(\infty)=\infty$ in the case when $\alpha=\infty$. It is worth mentioning that the spaces $\Lambda_{p,w}$ were introduced by Lorentz in \cite{Loren} and the space $\Lambda_{p,w}$ is a norm space (resp. quasi-norm space) if and only if $1\leq{p}<\infty$ and $w$ is decreasing, see \cite{KamMal} (resp. $W$ satisfies the condition $\Delta_2$, see \cite{Haaker,KamMal}). It is also known that for any $0<p<\infty$ if $W$ satisfies the condition $\Delta_2$ and $W(\infty) = \infty$, then the Lorentz space $\Lambda_{p,w}$ is an order continuous r.i. quasi-Banach function space (see \cite{KamMal}). 

For $0<p<\infty $ and $w\in L^{0}$ a nonnegative locally integrable weight function we consider the
Lorentz space $\Gamma _{p,w}$, that is a subspace of $L^{0}$ such that 
\begin{equation*}
\Vert x\Vert _{\Gamma _{p,w}}=\norm{x^{**}}{\Lambda_{p,w}}{}=\left( \int_{0}^{\alpha }x^{\ast \ast
	p}(t)w(t)dt\right) ^{1/p}<\infty .
\end{equation*}
Unless we say otherwise, we suppose that $w$ belongs to the class $D_{p}$, i.e. 
\begin{equation*}
W(s):=\int_{0}^{s}w(t)dt<\infty \mathnormal{\ \ \ }\text{\textnormal{and}}%
\mathnormal{\ \ }W_{p}(s):=s^{p}\int_{s}^{\alpha }t^{-p}w(t)dt<\infty
\end{equation*}%
for all $0<s\leq 1 $ if $\alpha =1 $ and for all $0<s<\infty$ otherwise. It is easy to observe that if $w\in{}D_p$, then the Lorentz space $\Gamma _{p,w}$ is nontrivial. Moreover, it is clear that $\Gamma _{p,w}\subset \Lambda _{p,w}.$ On the other hand, the following inclusion $\Lambda _{p,w}\subset \Gamma _{p,w}$ holds if and only if $w\in B_{p}$ (see \cite{KMGam}). Let us also recall that $\left( \Gamma _{p,w},\Vert \cdot \Vert_{\Gamma _{p,w}}\right)$ is a r.i. quasi-Banach function space with the Fatou property and was introduced by Calder\'{o}n in \cite{Cal}. It is well known that in the case when $\alpha =\infty $ the Lorentz space $\Gamma _{p,w}$ has order continuous norm if and only if $\int_{0}^{\infty }w\left( t\right) dt=\infty$ (see  \cite{KMGam}). It is also well known that by the Lions-Peetre $K$-method (see \cite{BruKru,KPS}), the space $\Gamma_{p,w}$ is an interpolation space between $L^{1}$ and $L^{\infty }$. For more details about the properties of the spaces $\Lambda_{p,w}$ and $\Gamma _{p,w}$ the reader is referred to \cite{Cies-geom,CieKolPan,CieKolPlu,KMGam,KamMal}. 

\section{lower and upper local uniform $K$-monotonicity in symmetric spaces}

In this section we investigate a connection between lower local uniform $K$-monotonicity and lower local uniform monotonicity in symmetric spaces. We also present a complete characterization of an $LLUKM$ point in terms of a point of order continuity and an $LKM$ point.

\begin{lemma}\label{lem:1:llukm}
	Let $E$ be a symmetric space. If $x\in{E}$ is $LLUKM$ point, then $x^*(\infty)=0$. 
\end{lemma}

\begin{proof}
	Suppose for a contrary that $x^*(\infty)>0$. Define $x_n=x^*\chi_{[0,n]}$ for any $n\in\mathbb{N}$. Then, for any $n\in\mathbb{N}$ we have $0\leq{x_n}\leq{x^*}$ and also $x_n\prec{x}.$ It is clear that $x_n\uparrow{}x^*$ a.e. and $\sup_{n\in\mathbb{N}}\norm{x_n}{E}{}\leq\norm{x}{E}{}<\infty$. Hence, by the Fatou property we conclude that $\norm{x_n}{E}{}\rightarrow\norm{x}{E}{}$. Consequently, by assumption that $x$ is $LLUKM$ point it follows that 
	\begin{equation*}
	\norm{x_n^*-x^*}{E}{}\rightarrow{0}.
	\end{equation*}
	Since $x^*(\infty)>0$ we obtain $\chi_{I}\in{E}$, whence for any $n\in\mathbb{N}$, $$\norm{x_n^*-x^*}{E}{}=\norm{x^*\chi_{(n,\infty)}}{E}{}\geq\norm{x^*(\infty)\chi_{(n,\infty)}}{E}{}=x^*(\infty)\norm{\chi_{I}}{E}{}>0.$$
	So, we get a contradiction which finishes the proof.
\end{proof}

\begin{lemma}\label{lem:LLUKM=>OC}
	Let $E$ be a symmetric space and $\phi$ be the fundamental function of $E$. If $x\in{E}$ is an $LLUKM$ point and $x^*(t)\phi(t)\rightarrow{0}$ as $t\rightarrow{0^+}$, then $x$ is a point of order continuity.
\end{lemma}

\begin{proof}
   	 Let us assume for a contrary that $x$ is not a point of order continuity in $E$. Then, by Lemma 2.6 \cite{CieKolPan} and Proposition 3.2 \cite{BS} there exist $(A_n)\subset{I}$ a decreasing sequence of measurable sets and $\delta>0$ such that $A_n\rightarrow\emptyset$ and
   	 \begin{equation}\label{equ:1:lem:LLUKM=>OC}
   	 \delta\leq\norm{x^*\chi_{A_n}}{E}{}
   	 \end{equation} 
   	 for all $n\in\mathbb{N}$. Let $\epsilon\in(0,\delta)$. We claim that there exists $K\in\mathbb{N}$ such that for every $k\geq{K}$,
   	\begin{equation*}
   	\norm{x^*\chi_{[k,\infty)}}{E}{}<\frac{\epsilon}{2}.
   	\end{equation*}
   	Indeed, taking $x_n=x^*\chi_{[0,n)}$ for any $n\in\mathbb{N}$ we have $x_n=x_n^*\uparrow{x^*}$ and also $\sup_{n\in\mathbb{N}}\norm{x_n^*}{E}{}\leq\norm{x^*}{E}{}<\infty$. Hence, by the Fatou property and by symmetry of $E$, it follows that $\norm{x_n}{E}{}\rightarrow\norm{x}{E}{}$. Consequently, according to assumption that $x$ is an $LLUKM$ point, in view of $x_n\prec{x}$ we obtain our claim. Moreover, it is easy to notice that $x^*\chi_{A_n\cap[0,k)}\prec{x^*}\chi_{[0,\min\{\mu(A_n),k\})}$ for any $k,n\in\mathbb{N}$, whence by symmetry and by the triangle inequality of the norm in $E$ we conclude
	\begin{align*}
	\norm{x^*\chi_{A_n}}{E}{}&\leq\norm{x^*\chi_{A_n\cap[0,k)}}{E}{}+\norm{x^*\chi_{A_n\cap[k,\infty)}}{E}{}\\
	&\leq\norm{x^*\chi_{[0,\min\{\mu(A_n),k\})}}{E}{}+\norm{x^*\chi_{A_n\cap[k,\infty)}}{E}{}	
	\end{align*}
	for any $k,n\in\mathbb{N}$. Hence, since $\mu(A_n)<K$ for sufficiently large $n\in\mathbb{N}$, passing to subsequence and relabelling if necessary, by the claim and by condition \eqref{equ:1:lem:LLUKM=>OC} we get 
	\begin{equation*}
	\delta\leq\norm{x^*\chi_{A_n}}{E}{}\leq\norm{x^*\chi_{[0,\mu(A_n))}}{E}{}+\norm{x^*\chi_{A_n\cap[K,\infty)}}{E}{}\leq\norm{x^*\chi_{[0,\mu(A_n))}}{E}{}+\frac{\epsilon}{2}
	\end{equation*}
	for any $n\in\mathbb{N}$. Therefore, for any $n\in\mathbb{N}$ we have
	\begin{equation}\label{equ:2:lem:LLUKM=>OC}
	\frac{\delta}{2}\leq\norm{x^*\chi_{[0,\mu(A_n))}}{E}{}.
	\end{equation}
	Define $t_n=\mu(A_n)$ and $z_n=x^*(t_n)\chi_{[0,t_n)}+x^*\chi_{[t_n,\infty)}$ for all $n\in\mathbb{N}$. Clearly, $z_n=z_n^*\leq{x^*}$ for every $n\in\mathbb{N}$ and $z_n^*\uparrow{x^*}$ a.e. on $I$. In consequence, since $\sup_{n\in\mathbb{N}}\norm{z_n^*}{E}{}\leq\norm{x^*}{E}{}$, by the Fatou property and by symmetry of $E$ this yields $\norm{z_n}{E}{}\rightarrow\norm{x}{E}{}$. Hence, since $z_n\prec{x}$ for any $n\in\mathbb{N}$ and by assumption that $x$ is an $LLUKM$ point there exists $N\in\mathbb{N}$ such that for any $n\geq{N}$,  
	\begin{equation*}
	\norm{(x^*-x^*(t_n))\chi_{[0,t_n)}}{E}{}<\frac{\epsilon}{4}.
	\end{equation*}
	So, by condition \eqref{equ:2:lem:LLUKM=>OC} and by the triangle inequality of the norm in $E$ we obtain
	\begin{align*}
	\frac{\delta}{2}&\leq\norm{x^*\chi_{[0,t_n)}}{E}{}\leq\norm{(x^*-x^*(t_n))\chi_{[0,t_n)}}{E}{}+\norm{x^*(t_n)\chi_{[0,t_n)}}{E}{}\\
	&\leq\frac{\epsilon}{4}+x^*(t_n)\phi(t_n)
	\end{align*}
	for all $n\geq{N}$. Consequently, for any $n\geq{N}$ we have
	\begin{equation}\label{equ:3:lem:LLUKM=>OC}
	x^*(t_n)\phi(t_n)\geq\delta/4,
	\end{equation}
	whence by assumption $x^*(t)\phi(t)\rightarrow{0}$ as $t\rightarrow{0^+}$ we get a contradiction, which ends the proof.
\end{proof}

Now, we answer the crucial question whether the condition $\phi(t)x^*(t)\rightarrow{0}$ as $t\rightarrow{0^+}$ in Lemma \ref{lem:LLUKM=>OC} is necessary and whether it can be avoided. Namely, in the following example we provide a function, in the Lorentz space $\Lambda_{1,\psi'}\cap{L^\infty}$, that is an $LLUKM$ point and it is not a point of order continuity.   

\begin{example}
	Let $\psi$ be a strictly concave function such that $\psi(0^+)=0$ and $\psi(\infty)=\infty$. Consider $E=\Lambda_{1,\psi'}\cap{L^\infty}$ on $I=[0,1]$, equipped with an equivalent norm given by 
	\begin{equation*}
	\norm{x}{E}{}=\norm{x}{\Lambda_{1,\psi'}}{}+\norm{x}{L^\infty}{}
	\end{equation*}
	for any $x\in{E}$. Assuming that $\phi$ is the fundamental function of $E$ we easily observe $\phi(t)=\psi(t)+1$ for any $t>0$. Define $x(t)=(1-t)\chi_{[0,1]}(t)$ for any $t\in{I}$. First, we prove that the function $x$ is not a point of order continuity in $E$. Indeed, taking $x_n=x\chi_{(0,1/n)}$ for any $n\in\mathbb{N}$ it is easy to see that $x_n\rightarrow{0}$ a.e. and $x_n\leq{x}$ for any $n\in\mathbb{N}$. Next, since $\lim_{t\rightarrow{0^+}}\phi(t)x^*(t)={1}$, by Proposition 5.9 in \cite{BS} we have
	$$\norm{x_n}{E}{}\geq\norm{x_n}{M_{\phi}}{}\geq\sup_{t\in(0,1/n]}\{(1-t)(1+\psi(t))\}\geq{1}$$
	for all $n\in\mathbb{N}$. We claim that $x$ is an LLUKM point in $E$. Since $\psi(\infty)=\infty$ and $\psi(0^+)=0$, by Proposition 1.4 in \cite{KMGam} it follows that the Lorentz space $\Lambda_{1,\psi'}$ is order continuous. Hence, since $\psi$ is strictly concave, by Theorem 2.11 in \cite{ChDSS} we obtain that $\Lambda_{1,\psi'}$ is strictly $K$-monotone and also $ULUKM$. Consequently, by Theorem \ref{thm:complete:charac} we conclude $\Lambda_{1,\psi'}$ is $LLUKM$. Hence, the Lorentz space $E$ endowed with the given norm is strictly $K$-monotone, whence $x$ is an $LKM$ point in $E$. Assume that $(y_n)\subset{E}$, $y_n\prec{x}$ for any $n\in\mathbb{N}$ and $\norm{y_n}{E}{}\rightarrow\norm{x}{E}{}$. Then, since $x$ is an $LKM$ point and $x^*(\infty)=0$, by Theorem 3.2 in \cite{Cies-SKM-KOC} it follows that $y_n^*\rightarrow{x^*}$ globally in measure. Therefore,  by property 2.11 in \cite{KPS} we get $y_n^*(t)\rightarrow{x^*(t)}$ for all $t\in[0,1]$. In consequence, by monotonicity of the decreasing rearrangement $y_n^*$ and by continuity of $x^*$ on $I$, in view of Dini's theorem for monotone functions (see \cite{Poly}) it follows that $y_n^*$ converges to $x^*$ uniformly on $I$, i.e. 
	\begin{equation}\label{equ:inf:conv}
	\norm{x^*-y_n^*}{L^\infty}{}\rightarrow{0}.
	\end{equation}
	So, it is clear that $$\norm{y_n}{L^\infty}{}=y_n^*(0)\rightarrow{}x^*(0)=\norm{x}{L^\infty}{}.$$
	Furthermore, by assumption $\norm{y_n}{E}{}\rightarrow\norm{x}{E}{}$ and by definition of the norm in $E$ we get $\norm{y_n}{\Lambda_{1,\psi'}}{}\rightarrow\norm{x}{\Lambda_{1,\psi'}}{}$. Thus, since $y_n\prec{x}$ for all $n\in\mathbb{N}$ and by the fact that $\Lambda_{1,\psi'}$ is $LLUKM$ we have
	\begin{equation*}
	\norm{x^*-y_n^*}{\Lambda_{1,\psi'}}{}\rightarrow{0}.
	\end{equation*} 
	and consequently, in view of condition \eqref{equ:inf:conv} and by definition of the norm in $E$ we are done.
\end{example}

\begin{proposition}\label{prop:LLUKM=>OC}
	Let $E$ be a symmetric space. If ${E}$ is $LLUKM$, then $E$ is order continuous.
\end{proposition}

\begin{proof}
	For a contrary, suppose that there exists $x\in{E}$ that is not a point of order continuity. Let $\phi$ be the fundamental function of $E$. By symmetry of $E$ and by Proposition 5.9 in  \cite{BS} we have for any $t>0$ and $z\in{E}$,
	\begin{equation}\label{equ:1:prop:LLUKM=>OC}
	z^*(t)\phi(t)\leq\norm{z}{M_\phi}{}\leq\norm{z}{E}{}.
	\end{equation}
	Next, proceeding similarly as in the proof of  Lemma \ref{lem:LLUKM=>OC}, in view of conditions \eqref{equ:3:lem:LLUKM=>OC} and \eqref{equ:1:prop:LLUKM=>OC} it is easy to see that
	\begin{equation*}
	 \frac{\delta}{4}\leq{}\norm{x}{L^\infty}{}\phi(0^+)\leq\norm{x}{E}{}.
	\end{equation*}
	Then, since $\phi(0^+)>0$, applying condition \eqref{equ:1:prop:LLUKM=>OC} for any $z\in{E}$ we observe
	\begin{equation}\label{equ:2:prop:LLUKM=>OC}
	\norm{z}{L^\infty}{}\phi(0^+)\leq\norm{z}{E}{}.
    \end{equation}
	Define $y=\chi_{[0,1)}$ and $y_n=\chi_{[0,1-1/n)}$ for any $n\in\mathbb{N}$. Obviously, by the Fatou property we get $\norm{y_n}{E}{}\rightarrow\norm{y}{E}{}$. Thus, since $y_n\prec{y}$ for all $n\in\mathbb{N}$, in view of assumption that $E$ is $LLUKM$ we get $$\norm{\chi_{[0,1/n)}}{E}{}=\norm{y^*-y_n^*}{E}{}\rightarrow{0}.$$
	Hence, by condition \eqref{equ:2:prop:LLUKM=>OC} we obtain a contradiction and complete the proof.
\end{proof}

\begin{theorem}\label{thm:1:LLUKM}
	Let $E$ be a symmetric space and $\phi$ be the fundamental function of $E$. If $x\in{E}$ is an $LLUKM$ point and $\lim_{t\rightarrow{0^+}}x^*(t)\phi(t)=0$, then $\abs{x}{}{}$ is an $LLUM$ point.
\end{theorem}

\begin{proof}
	Let $(x_n)\subset{E^+}$ and $0\leq{x_n}\leq{\abs{x}{}{}}$, $\norm{x_n}{E}{}\rightarrow\norm{x}{E}{}$. Then, by property of the maximal function we obtain $x_n\prec{x}$. Hence, by assumption that $x$ is an $LLUKM$ point we have 
	\begin{equation}\label{equ:1:thm:1:LLUKM}
	\norm{x_n^*-x^*}{E}{}\rightarrow{0}.
	\end{equation}
	By Lemma \ref{lem:1:llukm} we get $x^*(\infty)=0$, whence by Lemma 2.7 in \cite{CieKolPan} and by assumption that $0\leq{x_n}\leq\abs{x}{}{}$ for all $n\in\mathbb{N}$ it follows that $x_n$ converges to $|x|$ in measure. Moreover, since $\lim_{t\rightarrow{0^+}}x^*(t)\phi(t)=0$, by Lemma \ref{lem:LLUKM=>OC} this yields that $x$ is a point of order continuity. Consequently, by condition \eqref{equ:1:thm:1:LLUKM} and by Proposition 2.4 in \cite{CzeKam} we conclude $$\norm{x_n-\abs{x}{}{}}{E}{}\rightarrow{0}.$$
\end{proof}

\begin{theorem}\label{thm:LLUKM<=>LKM&OC}
Let $E$ be a symmetric space on $I=[0,1)$, with $\phi$ the fundamental function of $E$. A point $x\in{E}$ is an $LLUKM$ point and $\lim_{t\rightarrow{0^+}}x^*(t)\phi(t)=0$ if and only if $x$ is an $LKM$ point and a point of order continuity.
\end{theorem}

\begin{proof}
\textit{Necessity}. Immediately, by Remark 3.1 in \cite{Cies-geom} and by Lemma  \ref{lem:LLUKM=>OC} we complete the proof.\\
\textit{Sufficiency}. Let $(x_n)\subset{E}$, $x_n\prec{x}$ and $\norm{x_n}{E}{}\rightarrow\norm{x}{E}{}$. Since $x$ is a point of order continuity, it is easy to see that $\lim_{t\rightarrow{0^+}}x^*(t)\phi(t)=0$ and by Lemma 2.5 \cite{CieKolPan} it follows $x^*(\infty)=0$. Moreover, since $x$ is an $LKM$ point, by Theorem 3.2 in \cite{Cies-SKM-KOC} we obtain $x_n^*$ converges to $x^*$ in measure. Hence, by property 2.11 in \cite{KPS}, we get
\begin{equation}\label{equ:conv:pointwise}
(x_n^*-x^*)^+\rightarrow{0}\quad\textnormal{and}\quad(x^*-x_n^*)^+\rightarrow{0}
\end{equation}
a.e.  and in measure on $I$. Notice that, for any $n\in\mathbb{N}$ we have
\begin{equation}\label{equ:ineq+}
(x_n^*-x^*)^+\leq{x_n^*}\quad\textnormal{and}\quad(x^*-x_n^*)^+\leq{\sup_{k\geq{n}}(x^*-x_k^*)^+}\leq{x^*}
\end{equation}
a.e. on $I$. In consequence, since $\sup_{k\geq{n}}(x^*-x_n^*)^+\downarrow{0}$ a.e. and $x$ is a point of order continuity, by Lemma 2.6 in \cite{CieKolPan} we obtain
\begin{equation*}
\norm{(x^*-x_n^*)^+}{E}{}\rightarrow{0}.
\end{equation*}
Thus, by the triangle inequality of the norm in $E$, to complete the proof it is enough to show the following condition
\begin{equation}\label{equ:second:conv}
\norm{(x_n^*-x^*)^+}{E}{}\rightarrow{0}.
\end{equation}
First, by Lemma 3.1 \cite{Cies-JAT} it is clear that $x^{**}(\infty)=0$. Therefore, since ${x_n^*}\prec{x^*}$ for all $n\in\mathbb{N}$, by condition \eqref{equ:ineq+} it is easy to observe that for any $n\in\mathbb{N}$, 
\begin{equation}\label{equ:ineq:star}
((x_n^*-x^*)^+)^*\leq{x_n^*}\leq{x^{**}}\quad\textnormal{and}\quad(x_n^*-x^*)^+\prec{x^*,}
\end{equation}
whence, by condition \eqref{equ:conv:pointwise} and  by property 2.12 in \cite{KPS} we conclude
\begin{equation}\label{equ:conv:seq:star}
((x_n^*-x^*)^+)^*\rightarrow{0}
\end{equation}
pointwise and also in measure. Furthermore, by condition \eqref{equ:ineq:star} and by Hardy's lemma \cite{BS} for any $y\in{E}$ and $t>0$, $n\in\mathbb{N}$ we have
\begin{equation}\label{equ:hardy:ineq}
\int_{0}^{t}((x_n^*-x^*)^+)^*y^*\leq\int_{0}^{t}x^*y^*.
\end{equation}
Define for any $n, k\in\mathbb{N}$, $$M_n^k=\left\{t\in{I}:((x_n^*-x^*)^+)^*(t)>\frac{1}{k}\right\}.$$
Clearly, by condition \eqref{equ:conv:seq:star} for any $k\in\mathbb{N}$ we have $\mu(M_n^k)\rightarrow{0}$ as $n\rightarrow\infty$. Now, letting $y=\chi_{M_n^k}\in{E}$, by condition \eqref{equ:hardy:ineq} and by symmetry of $E$, in view of Corollary 4.7 in \cite{BS} we get 
\begin{equation*}
\norm{((x_n^*-x^*)^+)^*\chi_{[0,\mu(M_n^k)]}}{E}{}\leq\norm{x^*\chi_{[0,\mu(M_n^k)]}}{E}{}
\end{equation*}
for every $n,k\in\mathbb{N}$. Thus, since $x^*\chi_{[0,\mu(M_n^k)]}\leq{x^*}$ a.e. on $I$ for all $n,k\in\mathbb{N}$ and $x^*$ is a point of order continuity, it follows that for any $k\in\mathbb{N}$ and $\epsilon>0$ there exists $N\in\mathbb{N}$ such that for any $n\geq{N}$,
\begin{equation*}
\norm{((x_n^*-x^*)^+)^*\chi_{[0,\mu(M_n^k)]}}{E}{}\leq\frac{\epsilon}{2}.
\end{equation*}
Moreover, by construction of the set $M_n^k$, picking $k\in\mathbb{N}$ such that $\norm{\chi_I}{E}{}/k<\epsilon/2$ it is easy to see that
\begin{equation*}
\norm{((x_n^*-x^*)^+)^*\chi_{(\mu(M_n^k),1)}}{E}{}\leq\norm{\frac{1}{k}\chi_{(\mu(M_n^k),1)}}{E}{}\leq\frac{\epsilon}{2}
\end{equation*}
for all $n\in\mathbb{N}$. Finally, by the triangle inequality of the norm in $E$ we prove condition \eqref{equ:second:conv} and finish the proof.
\end{proof}

Now, we investigate a similar result as above for a symmetric space $E$ on $[0,\infty)$ under some additional assumptions of $E$.  

\begin{theorem}\label{thm:LLUKM<=>LKM&OC_2}
Let $E$ be a symmetric space on $I=[0,\infty)$ and let $\phi$ be the fundamental function of $E$ such that $\phi(t)/t\rightarrow{0}$ as $t\rightarrow\infty$ and let $x\in{E\cap{L^1}}$. A point $x$ is an $LLUKM$ point and $\lim_{t\rightarrow{0^+}}x^*(t)\phi(t)=0$ if and only if $x$ is an $LKM$ point and a point of order continuity.
\end{theorem}

\begin{proof}
Notice that proceeding analogously as in the proof of Theorem \ref{thm:LLUKM<=>LKM&OC} in sufficiency it is enough to show condition \eqref{equ:second:conv}. First, let us mention that by Lemma 2.5 in \cite{CieKolPan} and by Lemma 3.1 in \cite{Cies-JAT} and in view of the assumption $x$ is a point of order continuity it follows that $x^*(\infty)=x^{**}(\infty)=0$. Let $\epsilon>0$ and $t_\epsilon=d_{x^*}(\epsilon)$. Then, it is clear that $t_\epsilon<\infty$, and so by monotonicity of the decreasing rearrangement $x^*$ we obtain $x^*(t)\leq{\epsilon}$ for all $t\geq{t_\epsilon}$. For simplicity of our notation let us assume that $y_n=(x_n^*-x^*)^+$ for any $n\in\mathbb{N}$. First we claim that 
\begin{equation}\label{equ:first:part:conv}
\norm{y_n^*\chi_{[0,t_\epsilon)}}{E}{}\rightarrow{0}.
\end{equation}
Define a set
\begin{equation*}
	A_n=\left\{t\in[0,t_\epsilon]:x^*(t)\leq{y_n^*}(t)\right\}
\end{equation*}
for every $n\in\mathbb{N}$. Then, by monotonicity of $x^*$, it is easy to see that $x^*(t)\geq{\epsilon}$ for any $t\leq{t_\epsilon}$. Next, in view of condition \eqref{equ:conv:seq:star} we observe
\begin{equation}\label{equ:measu:set:conv}
\mu\left(A_n\right)\leq\mu\left(t\in[0,t_\epsilon]:y_n^*(t)\geq{\epsilon}\right)\rightarrow{0}.
\end{equation}
Moreover, by condition \eqref{equ:hardy:ineq} we obtain
\begin{equation*}
	\int_{0}^{t}y_n^*\chi_{[0,\mu(A_n)]}\leq\int_{0}^{t}x^*\chi_{[0,\mu(A_n)]}
\end{equation*}
for all $n\in\mathbb{N}$ and $t>0$. Hence, by Proposition 1.1 in \cite{ChSu} for any $t>0$ and $n\in\mathbb{N}$ we get
\begin{align*}
\left(y_n^*\chi_{A_n}\right)^{**}(t)&=\frac{1}{t}\int_{0}^{t}\left(y_n^*\chi_{A_n}\right)^*\leq\frac{1}{t}\int_{0}^{t}y_n^*\chi_{[0,\mu(A_n)]}\\
&\leq\left(x^*\chi_{[0,\mu(A_n)]}\right)^{**}(t)\leq{x}^{**}(t).
\end{align*}
Thus, by symmetry of $E$ we conclude
\begin{align*}
\norm{y_n^*\chi_{[0,t_\epsilon)}}{E}{}&\leq\norm{y_n^*\chi_{A_n}}{E}{}+\norm{y_n^*\chi_{[0,t_\epsilon)\setminus{}A_n}}{E}{}\\
&\leq\norm{x^*\chi_{[0,\mu(A_n)]}}{E}{}+\norm{y_n^*\chi_{[0,t_\epsilon)\setminus{}A_n}}{E}{}
\end{align*}
for each $n\in\mathbb{N}$. Consequently, since $y_n^*\chi_{[0,t_\epsilon)\setminus{}A_n}\leq{x^*}$ for any $n\in\mathbb{N}$, by conditions \eqref{equ:conv:seq:star} and \eqref{equ:measu:set:conv} as well as by assumption that $x$ is a point of order continuity and in view of Lemma 2.6 in \cite{CieKolPan} we prove our claim \eqref{equ:first:part:conv}. Now, without loss of generality passing to subsequence and relabelling we may assume that $y_n^*(t_\epsilon)>0$ for all $n\in\mathbb{N}$, because otherwise in view of the claim \eqref{equ:first:part:conv} we finish the proof. Furthermore, by condition \eqref{equ:ineq:star} and by assumption that $x\in{}E\cap{L^1}$ it is easy to notice that
\begin{equation*}	\int_{t_\epsilon}^{\infty}y_n^*\leq\int_{0}^{\infty}y_n^*\leq\int_{0}^{\infty}x^*<\infty
\end{equation*}
for all $n\in\mathbb{N}$. Denote for any $n\in\mathbb{N}$,
\begin{equation*}
\delta_n=t_\epsilon+\frac{1}{y_n^*(t_\epsilon)}\int_{t_\epsilon}^{\infty}y_n^*\quad\textnormal{and}\quad{}z_n=y_n^*\chi_{[0,t_\epsilon)}+y_n^*(t_\epsilon)\chi_{[t_\epsilon,\delta_n)}.
\end{equation*}
Now, we prove that 
\begin{equation}\label{equ:rest:conv}
\norm{y_n^*(t_\epsilon)\chi_{[t_\epsilon,\delta_n)}}{E}{}\rightarrow{0}.
\end{equation}
Assume for a contrary that $a=\inf_{n\in\mathbb{N}}\norm{y_n^*(t_\epsilon)\chi_{[t_\epsilon,\delta_n)}}{E}{}>0$. Then, passing to subsequence and relabelling if necessary we obtain 
	$$\norm{y_n^*(t_\epsilon)\chi_{[t_\epsilon,\delta_n)}}{E}{}\downarrow{a}.$$
	Hence, for any $n\in\mathbb{N}$ we notice that
	\begin{align*}
	a\leq\norm{y_n^*(t_\epsilon)\chi_{[t_\epsilon,\delta_n)}}{E}{}=&y_n^*(t_\epsilon)\phi(\delta_n-t_\epsilon)\\
	=&y_n^*(t_\epsilon)\phi\left(\frac{1}{y_n^*(t_\epsilon)}\int_{t_\epsilon}^{\infty}y_n^*\right)\\
	\leq&y_n^*(t_\epsilon)\phi\left(\frac{1}{y_n^*(t_\epsilon)}\int_{0}^{\infty}x^*\right).
	\end{align*}
	Therefore, letting $s_n=\int_{0}^{\infty}x^*/{y_n^*(t_\epsilon)}$ for all $n\in\mathbb{N}$ we have
	\begin{align*}
	a\leq&\frac{\phi(s_n)}{s_n}\int_{0}^{\infty}x^*
	\end{align*}
for all $n\in\mathbb{N}$. According to condition \eqref{equ:conv:seq:star} we observe $y_n^*(t_\epsilon)\rightarrow{0}$ and so $s_n\rightarrow\infty$. In consequence, by assumption that $\phi(t)/t\rightarrow{0}$ as $t\rightarrow\infty$ we get a contradiction which provides condition \eqref{equ:rest:conv}. Now, we show that $y_n\prec{z_n}$ for all $n\in\mathbb{N}$. Obviously, $y_n^{**}=z_n^{**}$ on $[0,t_\epsilon]$ for each $n\in\mathbb{N}$. Moreover, for any $n\in\mathbb{N}$ and $t\in(t_\epsilon,\delta_n)$ we have
\begin{align*}
\int_{0}^{t}z_n^*=\int_{0}^{t_\epsilon}y_n^*+y_n^*(t_\epsilon)(t-t_\epsilon)\geq\int_{0}^{t_\epsilon}y_n^*+\int_{t_\epsilon}^{t}y_n^*=\int_{0}^{t}y_n^*
\end{align*}
and also for any $t\geq{\delta_n}$,
\begin{align*}
\int_{0}^{t}z_n^*=\int_{0}^{t_\epsilon}y_n^*+y_n^*(t_\epsilon)(\delta_n-t_\epsilon)=\int_{0}^{t_\epsilon}y_n^*+\int_{t_\epsilon}^{\infty}y_n^*\geq\int_{0}^{t}y_n^*.
\end{align*}
Therefore, by symmetry of $E$ we get $\norm{z_n}{E}{}\geq\norm{y_n}{E}{}$. Thus, by conditions \eqref{equ:first:part:conv} and \eqref{equ:rest:conv} and by the triangle inequality of the norm in $E$ we complete the proof.
\end{proof}

Immediately, in view of Remark 3.1 in \cite{Cies-geom}, by Proposition \ref{prop:LLUKM=>OC} and Theorems \ref{thm:LLUKM<=>LKM&OC} and \ref{thm:LLUKM<=>LKM&OC_2} we obtain the following results.
\begin{corollary}\label{coro1:LLUKM<=>OC&SKM}
	Let $E$ be a symmetric space on $I=[0,\alpha)$ with $\alpha<\infty$. The space $E$ is $LLUKM$ if and only if $E$ is strictly $K$-monotone and order continuous.
\end{corollary}

\begin{corollary}\label{coro2:LLUKM<=>OC&SKM}
	Let $E$ be a symmetric space on $I=[0,\infty)$ with the fundamental function $\phi$ such that $\phi(t)/t\rightarrow{0}$ as $t\rightarrow\infty$ and let $F\subset{E}$ be a symmetric sublattice that is embedded in $L^1[0,\infty)$. Then, the space $F$ is $LLUKM$ if and only if $F$ is strictly $K$-monotone and order continuous.
\end{corollary}

Now, we investigate a relation between lower local uniform $K$- monotonicity and the Kadec-Klee property for global convergence in measure. First, we show an example of a function in a symmetric space $E$ on $I=[0,\infty)$ that is a point of lower local uniform $K$-monotonicity but it is no $H_g$ point in $E$.
We also discuss in this example a symmetric space $E$ on $I=[0,1)$ that is lower local uniformly $K$-monotone but it does not have the Kadec-Klee property for global convergence in measure.  We recall Example 2.8 \cite{ChDSS} and modify to the case when $I=[0,\alpha)$, where $\alpha\leq\infty$. For the sake of the reader's convenience we present the details of the modified example. 

\begin{example}
Let $\delta>0$ and let  $\phi_1$, $\phi_2$ be strictly concave functions such that $$\phi_i(0)=\phi_i(0^+)=0\quad\textnormal{and}\quad{\phi_i(\infty)=\lim_{t\rightarrow\infty}\phi_i(t)=\infty}\quad\textnormal{for }i=1,2,$$
and also
$$\phi_2(1)>\phi_1(1)+\delta\quad\textnormal{and}\quad{\lim_{t\rightarrow{0}}\frac{\phi_2(t)}{\phi_1(t)}=\lim_{t\rightarrow\infty}\frac{\phi_i(t)}{t}=0}\quad\textnormal{for }i=1,2.$$
Consider the space $E=\Lambda_{1,\phi_1'}\cap\Lambda_{1,\phi_2'}$ with a norm given by
$$\norm{x}{E}{}=\max\{\norm{x}{\Lambda_{1,\phi_1'}}{},\norm{x}{\Lambda_{1,\phi_2'}}{}\}$$
for all $x\in{E}$. Since $\phi_i(\infty)=\infty$ for $i=1,2$ it follows that the symmetric space $E$ is order continuous (see \cite{ChDSS,KMGam}). Hence, since $\phi_1$ and $\phi_2$ are strictly concave, by Theorem 2.11 in \cite{ChDSS} we get $E$ is strictly $K$-monotone. Consequently, in case when $I=[0,1)$, by Corollary \ref{coro1:LLUKM<=>OC&SKM} we obtain $E$ is $LLUKM$. Define 
\begin{equation*}
x=\chi_{[0,1]}\quad\textnormal{and}\quad{}x_n=x+\frac{\delta}{\phi_1(\frac{1}{n})}\chi_{[0,\frac{1}{n})}
\end{equation*} 
for any $n\in\mathbb{N}$. Obviously, $x_n\rightarrow{x}$ in measure and 
\begin{equation*}
\norm{x_n}{E}{}=\frac{\delta\phi_2(\frac{1}{n})}{\phi_1(\frac{1}{n})}+\phi_2(1)\rightarrow\phi_2(1)=\norm{x}{E}{}.
\end{equation*}
On the other hand, we observe $\norm{x_n-x}{E}{}\geq\delta$ for any $n\in\mathbb{N}$, which concludes that $x$ is no $H_g$ point in $E$ and consequently $E$ does not have the Kadec-Klee property for global convergence in measure. However, since $x\in{L^1[0,\infty)}$, by Theorem \ref{thm:LLUKM<=>LKM&OC_2} we get $x$ is an $LLUKM$ point in the space $E$ on $I=[0,\infty)$.
\end{example}

\begin{theorem}\label{thm:LLUKM&Hg}
Let $E$ be a symmetric space and $x,x_n\in{E}$ with $x^*(\infty)=0$ and let:
\begin{itemize}
	\item[$(i)$] $x$ is an $LKM$ point and an $H_g$ point.
	\item[$(ii)$] $x$ is an $LKM$ point and $$x_n^{**}\rightarrow{x^{**}}\quad\textnormal{in measure,}\quad\norm{x_n}{E}{}\rightarrow\norm{x}{E}{}\quad\Rightarrow\quad\norm{x_n^*-x^*}{E}{}\rightarrow{0}.$$
	\item[$(iii)$] $x$ is an $LLUKM$ point.
\end{itemize}
Then, $(i)\Rightarrow(ii)\Rightarrow(iii)$. If $x$ is an $H_g$ point, then $(iii)\Rightarrow(i)$.
\end{theorem}

\begin{proof}
$(i)\Rightarrow(ii)$. Let $x,x_n\in{E}$ for any $n\in\mathbb{N}$, $x_n^{**}\rightarrow{x^{**}}$ in measure and $\norm{x_n}{E}{}\rightarrow\norm{x}{E}{}$. Now, proceeding analogously as in the proof of Theorem 3.8 \cite{Cies-geom}, under the assumption that $x$ is an $H_g$ point and $x^*(\infty)=0$, in view of Theorem 3.3 \cite{CieKolPlu} we complete the proof.\\
$(ii)\Rightarrow(iii)$. Let $x,x_n\in{E}$, $x_n\prec{x}$ for any $n\in\mathbb{N}$ and $\norm{x_n}{E}{}\rightarrow\norm{x}{E}{}$. Hence, by Theorem 3.2 in \cite{Cies-SKM-KOC} it follows that $x_n^{**}\rightarrow{x^{**}}$ in measure. Therefore, by condition $(ii)$ we get $\norm{x_n^*-x^*}{E}{}\rightarrow{0}$, which proves that $x$ is an $LLUKM$ point.\\
$(iii)\Rightarrow(i)$. Let $x$ be an $H_g$ point in $E$. Immediately, by Remark 3.1 in \cite{Cies-geom} we get $x$ is an $LKM$ point and this ends the proof.
\end{proof}

In the next example we present a symmetric space with the Kadec-Klee property for global convergence in measure which does not have $LLUKM$ property.

\begin{example}
Consider the Lorentz space $\Gamma_{p,w}$ with $0<p<\infty$ and $w$ a nonnegative weight function. If $W(\infty)<\infty$ or $W(t)=\int_{0}^{t}w$ is not strictly increasing, then by Proposition 1.4 in \cite{KMGam} or by Theorem 2.10 in \cite{CieKolPluSKM} respectively, we obtain the Lorentz space $\Gamma_{p,w}$ is not order continuous or it is not strictly $K$-monotone respectively. Moreover, we have $\lim_{t\rightarrow{0^+}}\norm{x^*\chi_{[0,t)}}{\Gamma_{p,w}}{}=0$ (see \cite{KMGam}), whence and by the monotonicity of the decreasing rearrangement $x^*$ we get $\lim_{t\rightarrow{0^+}}x^*(t)\phi(t)=0$, where $\phi$ is the fundamental function of $\Gamma_{p,w}$. In consequence,  by Remark 3.1 in \cite{Cies-geom} or by Lemma \ref{lem:LLUKM=>OC} respectively, it follows that $\Gamma_{p,w}$ is not $LLUKM$. On the other hand, by Theorem 4.1 in \cite{CieKolPlu} we know that the Lorentz space $\Gamma_{p,w}$ has the Kadec-Klee property for global convergence in measure. 
\end{example}

Now, we present the full characterization of lower and upper local uniform $K$ monotonicity in a symmetric space $E$ with order continuous norm. Next, we establish a correlation between upper local uniform $K$-monotonicity and upper local uniform monotonicity in $E$.

\begin{theorem}\label{thm:complete:charac}
	Let $E$ be a symmetric space with order continuous norm. Then, the following conditions are equivalent.
	\begin{itemize}
		\item[$(i)$] $E$ is $SKM$ and for any $(x_n)\subset{E}$, $x\in{E}$, $$x_n^{**}\rightarrow{x^{**}}\quad\textnormal{in measure and}\quad\norm{x_n}{E}{}\rightarrow\norm{x}{E}{}\quad\Rightarrow\quad\norm{x_n^*-x^*}{E}{}\rightarrow{0}.$$
		\item[$(ii)$] $E$ is $LLUKM$ and has the Kadec-Klee property for global convergence in measure.
		\item[$(iii)$] $E$ is $SKM$ and has the Kadec-Klee property for global convergence in measure.
		\item[$(iv)$] $E$ is $SKM$ and has the Kadec-Klee property for local convergence in measure.
		\item[$(v)$] $E$ is $ULUKM$.
	\end{itemize}
\end{theorem}

\begin{proof}
	It is well known that the equivalences $(iii)\Leftrightarrow(iv)\Leftrightarrow(v)$ follows directly from Theorem 2.7 in \cite{ChDSS}.  Immediately, by Theorem 3.8 in  \cite{Cies-geom} and by Theorem 3.5 in \cite{CieKolPlu} we get $(i)\Leftrightarrow(iii)\Leftrightarrow(v)$. Finally, the consequence of Lemma 2.5 in \cite{CieKolPan} and Theorem \ref{thm:LLUKM&Hg} is the following conclusion $(ii)\Leftrightarrow(iii)$.
\end{proof}

\begin{theorem}\label{}
	Let $E$ be a symmetric space. If $x\in{E}$ is a point of order continuity and a $ULUKM$ point, then $|x|$ is a $ULUM$ point and $x$ is an $H_g$ point.
\end{theorem}

\begin{proof}
	Let $(x_n)\subset{E^+}$, $|x|\leq{}x_n$ and $\norm{x_n}{E}{}\rightarrow\norm{x}{E}{}$. Then, by Proposition 3.2 in \cite{BS} we get $x\prec{x_n}$ for all $n\in\mathbb{N}$ and consequently by assumption that $x$ is a $ULUKM$ point we have $\norm{x_n^*-x^*}{E}{}\rightarrow{0}$. Hence, by the implication $(iii)\Rightarrow(ii)$ in the proof of Theorem 3.2 in \cite{ChDSS} it follows that $x_n$ converges to $|x|$ in measure. Consequently, by assumption that $x$ is a point of order continuity and by Proposition 2.4 in \cite{CzeKam} we have $\norm{x_n-|x|}{E}{}\rightarrow{0}$. Finally, in view of assumptions, by Theorem 3.8 in \cite{Cies-geom} and by Theorem 3.5 in \cite{CieKolPlu} we conclude $x$ is an $H_g$ point in $E$.
\end{proof}

In the next example we show that if the assumption $x$ is a point of order continuity of the above theorem is missing, then the implication is not true. 

\begin{example}
Take $E=L^\infty$ on $I=[0,\infty)$	and $x=\chi_{I}$. Let $(x_n)\subset{E}$ be such that $x\prec{x_n}$ for any $n\in\mathbb{N}$ and $\norm{x_n}{E}{}\rightarrow\norm{x}{E}{}$. Since $x^*=1$ on $I$, we claim that $x^*\leq{}x_n^*$ a.e for all $n\in\mathbb{N}$. Indeed, if it is not true, then there exist $(n_k)\subset\mathbb{N}$ and $(t_k)\subset{I}$ such that for any $k\in\mathbb{N}$ and $t\geq{t_k}$ we have
\begin{equation*}
x_{n_k}^*(t)\leq{}x_{n_k}^*(t_k)<1.
\end{equation*}
Hence, setting $k\in\mathbb{N}$ we observe for sufficiently large $t>t_k$,  $$x_{n_k}^{**}(t)<x^{**}(t)=1.$$
Therefore, by assumption $x\prec{x_n}$ for all $n\in\mathbb{N}$ we get a contradiction which proves our claim. It is easy to notice that $x$ is a $ULUM$ point in $E$ (see also \cite{CieKolPan}). Thus, according to the claim and by assumption $\norm{x_n^*}{E}{}\rightarrow\norm{x^*}{E}{}$ we obtain $$\norm{x_n^*-x^*}{E}{}\rightarrow{0}.$$
In consequence, we get $x$ is a $ULUKM$ point. On the other hand, taking $y_n=\chi_{(\frac{1}{n},\infty)}$ for any $n\in\mathbb{N}$, it is easy to see that $y_n\rightarrow{x}$ in measure and $\norm{y_n}{E}{}=\norm{x}{E}{}=1$ and also $\norm{x-y_n}{E}{}=1$ for every $n\in\mathbb{N}$. So, it follows that $x$ is no $H_g$ point in $E$. 
\end{example}

Now we discuss a correlation between $K$-order continuity and lower local uniform $K$-monotonicity in symmetric spaces. 

\begin{theorem}\label{thm1:KOC&LLUKM}
	Let $E$ be a symmetric space. If $x\in{E}$ is a point of $K$-order continuity and an $LKM$ point and also $x^*(\infty)=0$, then $x$ is an $LLUKM$ point.
\end{theorem}

\begin{proof}
	Let $(x_n)\subset{E}$ with $x_n\prec{x}$ for all $n\in\mathbb{N}$ and $\norm{x_n}{E}{}\rightarrow\norm{x}{E}{}$. 
	Observe that for each $n\in\mathbb{N}$,
	\begin{equation}\label{equ:1:thm1:KOC}
	(x^*-x_n^*)^+\leq{x^*}\quad\textnormal{and}\quad(x_n^*-x^*)^+\prec{}x_n^*\prec{x^*}.
	\end{equation}
	Moreover, since $x$ is $LKM$ point and $x^*(\infty)=0$, by assumption that $x_n\prec{x}$ for any $n\in\mathbb{N}$ and $\norm{x_n}{E}{}\rightarrow\norm{x}{E}{}$ and by Theorem 3.2 in \cite{Cies-SKM-KOC} it follows that $x_n^*$ converges to $x^*$ in measure. Hence, by property $2.11$ in \cite{KPS} we get $$\left((x_n^*-x^*)^+\right)^*\rightarrow{0}\quad\textnormal{and}\quad\left((x^*-x_n^*)^+\right)^*\rightarrow{0}$$ a.e. on $I$. In consequence, by condition \eqref{equ:1:thm1:KOC} and by assumption that $x$ is a point of $K$-order continuity we have
	\begin{equation*}
	\norm{((x^*-x_n^*)^+)^*}{E}{}\rightarrow{0}\quad\textnormal{and}\quad\norm{((x_n^*-x^*)^+)^*}{E}{}\rightarrow{0}.
	\end{equation*}
	Thus, by symmetry of $E$ and by the triangle inequality of the norm in $E$ we conclude $x_n^*$ converges to $x^*$ in norm of $E$.
\end{proof}

We present an example of a symmetric space having upper and lower local uniform $K$-monotonicity but not satisfying $K$-order continuity. 

\begin{remark}
Let $\psi(t)=t^{1/4}$ for any $t\in{I}$. Consider the space $E=\Lambda_{1,\psi'}\cap{}L^1$ on $I$ endowed with the equivalent norm given by $\norm{x}{E}{}=\norm{x}{\Lambda_{1,\psi'}}{}+\norm{x}{L^1}{}$. We claim that $(E,\norm{\cdot}{E}{})$ is $LLUKM$ and $ULUKM$, but it is not $KOC$. First denote $\phi(t)=\psi(t)+t$ for any $t\in{I}$. Observe that $E=\Lambda_{1,\phi'}$ and $\phi(t)/t\rightarrow{1}$ as $t\rightarrow\infty$. Define 
\begin{equation*}
x(t)=\chi_{[0,1)}(t)+\frac{1}{t^2}\chi_{[1,\infty)}(t)\quad\textnormal{and}\quad{}x_n(t)=\frac{1}{n}\chi_{[0,n)}(t)
\end{equation*}
for any $t>0$ and $n\in\mathbb{N}$. It is easy to see that $x=x^*$, $x_n=x_n^*\rightarrow{0}$ a.e. Clearly,
\begin{equation*}
x^{**}(t)=\chi_{[0,1)}(t)+\frac{2{t}-1}{t^2}\chi_{[1,\infty)}(t)
\end{equation*}
and
\begin{equation*}
{x_n^{**}(t)=\frac{1}{n}\chi_{[0,n)}(t)+\frac{1}{t}\chi_{[n,\infty)}(t)}
\end{equation*}
for any $t>0$ and  $n\in\mathbb{N}$, whence $x_n\prec{x}$ for all $n\in\mathbb{N}$. Notice that $x\in{E}$ and
\begin{equation*}
\norm{x_n}{E}{}=\norm{x_n}{\Lambda_{1,\psi'}}{}+\norm{x_n}{L^1}{}=1+\frac{1}{n^{3/4}}
\end{equation*}
for any $n\in\mathbb{N}$. Therefore, $\norm{x_n}{E}{}\geq{1}$ for every $n\in\mathbb{N}$, which concludes that $E$ is not $KOC$. On the other hand, since $\phi(\infty)=\int_{0}^{\infty}\phi'=\infty$, by Proposition 1.4 in \cite{KMGam} it follows that the Lorentz space $\Lambda_{1,\phi'}$ is order continuous. Hence, since $\phi$ is strictly concave, by Theorem 2.11 and Proposition 1.7 in \cite{ChDSS} we obtain that $\Lambda_{1,\phi'}$ is strictly $K$-monotone and also has the Kadec-Klee property for global convergence in measure. Finally, by Theorem \ref{thm:complete:charac} we get $E$ is $ULUKM$ and $LLUKM$.
\end{remark}

According to Theorem 4.8 in \cite{Cies-SKM-KOC} and by Remark 3.1 in \cite{Cies-geom} and also Lemma \ref{lem:LLUKM=>OC} as well as Theorem \ref{thm1:KOC&LLUKM} we conclude with the next theorem.

\begin{theorem}\label{thm:LLUKM&KOC&LKM}
Let $E$ be a symmetric space and let $\phi$ be the fundamental function of $E$ and $x\in{E}$. Then, the following conditions are equivalent.
\begin{itemize}
    \item[$(i)$]   $x$ is an $LLUKM$ point and $$\lim_{t\rightarrow{0^+}}\phi(t)x^{*}(t)=\lim_{s\rightarrow\infty}\phi(s)x^{**}(s)={0}.$$
	\item[$(ii)$]  $x$ is an $LKM$ point and a point of order continuity and $$\lim_{s\rightarrow\infty}\phi(s)x^{**}(s)={0}.$$
	\item[$(iii)$] $x$ is an $LKM$ point and a point of $K$-order continuity and $x^*(\infty)=0$.
\end{itemize} 
\end{theorem}

\subsection*{Acknowledgement}

We wish to express our gratitude to the reviewer for many valuable suggestions and remarks.

\bibliographystyle{amsplain}

\end{document}